\documentclass[oneside,english]{amsart}
\usepackage[T1]{fontenc}
\usepackage[latin9]{inputenc}
\usepackage{amsthm}
\usepackage{amstext}
\usepackage{amssymb}

\makeatletter
\numberwithin{equation}{section}
\numberwithin{figure}{section}
\theoremstyle{plain}
\newtheorem{thm}{\protect\theoremname}[section]
  \theoremstyle{definition}
  \newtheorem{defn}[thm]{\protect\definitionname}
  \theoremstyle{plain}
  \newtheorem{lem}[thm]{\protect\lemmaname}
  \theoremstyle{plain}
  \newtheorem{cor}[thm]{\protect\corollaryname}
  \theoremstyle{remark}
  \newtheorem{rem}[thm]{\protect\remarkname}
  \theoremstyle{definition}
  \newtheorem{example}[thm]{\protect\examplename}

\AtBeginDocument{

}

\makeatother

\usepackage{babel}
  \providecommand{\corollaryname}{Corollary}
  \providecommand{\definitionname}{Definition}
  \providecommand{\examplename}{Example}
  \providecommand{\lemmaname}{Lemma}
  \providecommand{\remarkname}{Remark}
\providecommand{\theoremname}{Theorem}

\begin{document}
\vspace{0.5in}

\title{Image partition near an Idempotent }

\author{Tanushree Biswas}

\address{Department of Mathematics, University of Kalyani, Kalyani-741235,
West Bengal, India}

\email{tanushreebiswas87@gmail.com}

\keywords{Ramsey theory, Central sets near an idempotent, Image partition near
an idempotent.}
\begin{abstract}
\textsl{\emph{Some of the classical results of Ramsey Theory can be
naturally stated in terms of image partition regularity of matrices.
There have been many significant results of image partition regular
matrices as well as image partition regular matrices }}\textsl{near
zero.}\textsl{\emph{ Here, we are investigating image partition regularity
near an idempotent of an arbitrary Hausdorff semitopological semigroup
$\left(T,\,+\right)$ and a dense subsemigroup $S$ of $T$. We describe
some combinatorial applications on finite as well as infinite image
partition regular matrices based on the Central Sets Theorem near
an idempotent of $T$. }}
\end{abstract}
\maketitle

\section{Introduction}

One of the most famous results of the field of Ramsey Theory was given
by van der Waerden in the year 1927, guaranteeing monochromatic arithmetic
progressions.
\begin{thm}[van der Waerden]
\label{thm:(van-der-Waerden).}Let $r\in\mathbb{N}$ and\textbf{
}let $\mathbb{N}={\displaystyle \bigcup_{i=1}^{r}A_{i}}$ be partitioned
into finitely many cells. Then, for any arbitrary length $l$, there
exist $i\in\{1,\,2,\,\ldots\,,\, r\}$ and $a,\, d\in\mathbb{N}$
such that 

\[
a,\, a+d,\,\ldots\,,\, ld\in A_{i}.
\]
\end{thm}
\begin{proof}
See \cite{waerden1927van}
\end{proof}
These above classical results of Ramsey Theory can be also represented
in a following way . Given $u,\, v\in\mathbb{N}$ and a $u\times v$
matrix $M$ with non-negative integer entries, whenever $\mathbb{N}$
is finitely colored there must exist some $\vec{x}\in\mathbb{N}^{v}$
such that the entries of $M\vec{x}$ are monochromatic. The arithmetic
progression $\left\{ a,\, a+d,\, a+2d,\, a+3d\right\} $ in the van
der Waerden's Theorem for $l=3$ is precisely the set of entries of
\[
\left(\begin{array}{cc}
1 & 0\\
1 & 1\\
1 & 2\\
1 & 3
\end{array}\right)\cdot\left(\begin{array}{c}
a\\
d
\end{array}\right).
\]
 Also, Schur's Theorem \cite{schur1916uber} and Hilbert's Theorem
\textbf{\cite{hilbert1933irreduzibilitat}} (for $k=3$) can also
be written in the above manner using the matrices 
\[
\left(\begin{array}{cc}
1 & 0\\
0 & 1\\
1 & 1
\end{array}\right)
\]
 
\[
\left(\begin{array}{cccc}
1 & 1 & 0 & 0\\
1 & 0 & 1 & 0\\
1 & 1 & 1 & 0\\
1 & 0 & 0 & 1\\
1 & 1 & 0 & 1\\
1 & 0 & 1 & 1\\
1 & 1 & 1 & 1
\end{array}\right)
\]
 respectively.

These matrices are called \emph{image partition regular matrices.
}We give a proper definition below.
\begin{defn}
Let $u,\, v\in\mathbb{N}$, and let $M$ be a $u\times v$ matrix
. We say $M$ to be an image partition regular matrix over $\mathbb{N}$
if and only if whenever $r\in\mathbb{N}$ and $\mathbb{N}=\cup_{i=1}^{r}A_{i}$,
there exist $i\in\left\{ 1,\,2,\,\ldots\,,\, r\right\} $ and $\vec{x}\in\mathbb{N}{}^{v}$
such that $M\vec{x}\in A_{i}^{u}$. 
\end{defn}
Image partition regularity is one of the most important concepts of
Ramsey Theory. Similarly, image partition regularity near zero has
been defined over a dense subsemigroup $\left(\left(0,\infty\right),\,+\right)$.
Further, in \cite{hindman1999semigroup}, some algebraic results were
obtained on the Stone-\v{C}ech compactification of the ultrafilters
near zero defined over a dense subsemigroup $\left(\left(0,\infty\right),\,+\right)$.
Based on the results obtained in \cite{hindman1999semigroup}, the
authors in \cite{de2009image} investigated various combinatorial
properties partition regularity near zero. In this article, we are
trying to generalize the notion of \emph{near zero }into an arbitrary
idempotent using the idea of Stone-\v{C}ech compactification. We first
give a brief description of algebraic structure of $\beta T$.

Let $\left(T,\,+\right)$, be a discrete space. We get the Stone-\v{C}ech
compactification of $T$ by taking the points of $\beta T$ to be
the ultrafilters on $T$, identifying the principal ultrafilters with
the points of $T$ and thus pretending that $T\subseteq\beta T$.
Given $A\subseteq T$, $A=\overline{A}=\{p\in\beta T:\, A\in p\}$
forms a basis of open (close) sets for a topology on $\beta T$. The
operation $+$ on $T$ has a natural extension to the Stone-\v{C}ech
compactification $\beta T$ of $T$ and under this extended operation,
$(\beta T,+)$ is a compact right topological semigroup (meaning that,
for any $p\in\beta T$, the function $\rho_{p}:\beta T\rightarrow\beta T$
defined by $\rho_{p}(q)=q+p$ is continuous) with $T$ contained in
its topological center (meaning that, for any $x\in T$, the function
$\lambda_{x}:\beta T\rightarrow\beta T$ defined by $\lambda_{x}(q)=x+q$
is continuous). Given $p,q\in\beta T$ and $A\subseteq T$, $A\in p+q$
if and only if $\{x\in T:-x+A\in q\}\in p$, where $-x+A=\{y\in T:x+y\in A\}$. 

The set $0^{+}$ of all nonprincipal ultrafilters on $T=\left(\left(0,\,\infty\right),\,+\right)$
that converge to $0$ ($0$ being the idempotent in $T=\left(\left(0,\,\infty\right),\,+\right)$
) is a semigroup under the restriction of the usual operation $'+'$
on $\beta T$ , the Stone-\v{C}ech compactification of the discrete
semigroup $T=\left(\left(0,\,\infty\right),\,+\right)$ . The set
$0^{+}$ is nothing but ultrafilters near $0$. In \cite{hindman1999semigroup},
N. Hindman and I. Leader characterized the smallest ideal of $\left(0^{+},\,+\right)$,
its closure, and those sets \emph{central} in $\left(0^{+},\,+\right)$
, that is, those sets which are members of minimal idempotents in
$\left(0^{+},\,+\right)$. They derived new combinatorial applications
of those sets that are central in $\left(0^{+},\,+\right)$. Further,
in \cite{de2009image} and \cite{de2009sets}, De, Hindman and Strauss
worked out some more and much more elaborated applications of central
sets.

Later, Akbari Tootkaboni, M and Vahed, T, in \cite{akbari2012semigroup}
defines a semigroup of ultrafilters near an arbitrary idempotent (instead
of taking the particular $0$). We give basic definitions and results
which will be required throughout the paper.
\begin{defn}
Let $\left(T,\,+\right)$ be a semitopological semigroup and $S$
be a dense subsemigroup of $T$. Let $E\left(T\right)$ be the set
consisting of idempotents of $T$. Then 
\begin{enumerate}
\item Given $e\in E(T)$, $e^{*}=\left\{ p\in\beta T_{d}\colon\, e\in\cap_{A\in p}\overline{A}\right\} $
\item Given $x\in T$, $x_{S}^{*}=\left\{ p\in\beta S_{d}\colon\, x\in\cap_{A\in p}\overline{A}\right\} $
\end{enumerate}
\end{defn}
In \cite{akbari2012semigroup}, Akbari Tootkaboni, M and Vahed, T,
defined ultrafilters near a point and showed that if $S$ is a dense
subsemigroup of a semitopological semigroup $\left(T,\,+\right)$and
if $e\in T$ is an idempotent, then $e_{S}^{*}$ forms a compact right
topological semigroup.

As a compact right topological semigroup, $e_{S}^{*}$ has a smallest
two-sided ideal \cite[Theorem 1.3.11]{berglund1989analysis} that
contains idempotents \cite[Corollary 2.10]{ellis1969lectures}. They
characterized the members of the smallest ideal of $\left(e^{*},\,+\right)$,
its closure and also those subsets of $T$ that have idempotents in
$\left(e^{*},\,+\right)$ in their closure. They also defined the
sets central near $e$, and derived various combinatorial applications
from it. 

Before moving further, we need some basic definitions and results
based on which our article stands. Note that throughout this paper,
$\left(T,\,+\right)$ denotes a Hausdorff semitopological semigroup
and $S$ is a dense subsemigroup of $T$. $E\left(T\right)$ denotes
the collection of all idempotents in $T$ . For every $x\in T$, $\tau_{x}$
denotes the collection of all neighborhoods of $x$ , where a set
$U$ is called a neighborhood of $x\in T$ if $x\in int_{T}\left(U\right)$.
For a set $A$, we write $\mathcal{P}_{f}\left(A\right)$ for the
set of finite nonempty subsets of $A$ and $\mathcal{P}\left(A\right)$
denote the collection of all subsets of $A$. 
\begin{defn}
Let $\left\{ x_{n}\right\} $ be a sequence in $S$. Then
\[
\mbox{FS}\left(\left\langle x_{n}\right\rangle _{n=1}^{\infty}\right)=\left\{ \sum_{n\in F}x_{n}\colon\, F\,\mbox{is\,\ a}\,\mbox{finite\,\ nonempty\,\ subset\,\ of}\,\mathbb{N}\right\} 
\]
is called \emph{finite sum }of the sequence\textbf{ $x_{n}$}.
\end{defn}

\begin{defn}
Let $\left(T,\,+\right)$ be a semitopological semigroup.
\begin{enumerate}
\item Let $\mathcal{B}$ be a local base at the point $x\in T$. We say
$\mathcal{B}$ has the finite cover property if $\left\{ V\in\mathcal{B}\colon\, y\in V\right\} $
is finite for each $y\in T\setminus\left\{ x\right\} $.
\item Let $S$ be a dense subsemigroup of $T$ and $e\in E\left(T\right)$.
Let $\left\{ x_{n}\right\} $ be a sequence in $S$. We say $\sum_{n=1}^{\infty}x_{n}$
converges near $e$ if for each $U\in\tau_{e}$ there exists $m\in\mathbb{N}$
such that $FS\left(\left\langle x_{n}\right\rangle _{n=k}^{l}\right)\subseteq U$
for each $l>k\geq m$.
\item Let $\mathcal{B}=\left\{ U_{n}\colon\, n\in\mathbb{N}\right\} $ be
a countable local base at the point $x\in T$ such that $U_{n+1}\subseteq U_{n}$
for each $n\in\mathbb{N}$, $U_{n+1}+U_{n+1}\subseteq U_{n}$, and
for each sequence $\left\{ x_{n}\right\} _{n=1}^{\infty}$ if $x_{n}\in U_{n}$
for each $n\in\mathbb{N}$ then $\sum_{n=1}^{\infty}x_{n}$ converges
near $x$. Then we say $\mathcal{B}$ is a countable local base for
convergence at the point $x\in T$ . 
\item Let $\mathcal{B}$ be a local base at the point $x\in T$ . If $\mathcal{B}$
satisfies in conditions $\textbf{\ensuremath{\left(a\right)}}$ and
$\textbf{\ensuremath{\left(c\right)}}$ then $\mathcal{B}$ is called
a countable local base that has the finite cover property for convergence
at the point $x$. For simplicity we say $\mathcal{B}$ has the \textbf{F}
property at the point $x$. 
\end{enumerate}
\end{defn}
$ $
\begin{thm}
Let $\left(T,\,+\right)$ be a semitopological semigroup and $S$
be a dense subsemigroup of $T$, $e\in E\left(T\right)$ and $\mathcal{B}=\left\{ U_{n}\right\} _{n=1}^{\infty}$
the \textbf{F} property at the point $e$. Then there exists $p=p+p$
in $e_{S}^{*}$ with $A\in p$ if and only if there is some sequence
$\left\langle x_{n}\right\rangle _{n=1}^{\infty}$ in $S$ such that
$\lim_{n\to\infty}x_{n}=e$, $\sum_{n=1}^{\infty}x_{n}$ converges
near $e$ and $FS\left(\left\langle x_{n}\right\rangle _{n=1}^{\infty}\right)\subseteq A$.\end{thm}
\begin{proof}
\cite[Theorem 3.2]{akbari2012semigroup}.\end{proof}
\begin{defn}
Let $S$ be a dense subsemigroup of $\left(T,\,+\right)$ and $e\in E\left(T\right)$.
\begin{enumerate}
\item $K$ is the smallest ideal of $e_{S}^{*}$.
\item A subset $B$ of $S$ is \emph{syndetic near $e$} if and only if
for each $U\in\tau_{e}$ , there exist some $F\in\mathcal{P}_{f}\left(U\cap S\right)$
and some $V\in\tau_{e}$ such that $S\cap V\subseteq\cup_{t\in F}\left(-t+B\right)$
. 
\item A set $A\subseteq S$ is central near $e$ if and only if there is
some idempotent $p\in K$ with $A\in p$. 
\end{enumerate}
\end{defn}

\begin{defn}
Let $\left(T,\,+\right)$ be a semitopological semigroup and $S$
be a dense subsemigroup of $T$. Let $\mathcal{B}=\left\{ U_{n}\right\} _{n=1}^{\infty}$has
the \textbf{F }property at the point $e\in E\left(T\right)$.
\begin{enumerate}
\item $\Phi=\left\{ f\colon\,\mathbb{N}\to\mathbb{N}\colon\,\forall n\in\mathbb{N},\, f\left(n\right)\leq n\right\} $.
\item $\mathcal{Y}=\left\langle \left\langle y_{i,\, t}\right\rangle _{t=1}^{\infty}\right\rangle _{i=1}^{\infty}$such
that for each $i\in\mathbb{N},\,\left\langle y_{i,\, t}\right\rangle _{t=1}^{\infty}\in S$
is a sequence for which $\sum_{t=1}^{\infty}y_{i,\, t}$ converges.
\item Given $Y=\left\langle \left\langle y_{i,\, t}\right\rangle _{t=1}^{\infty}\right\rangle _{i=1}^{\infty}\in\mathcal{Y}$
and $A\subseteq S$, $A$ is a $J_{Y}-set$ near $e$ if and only
if $\forall n\in\mathbb{N}$ there exist $a\in U_{n}$ and $H\in\mathcal{P}_{f}\left(\mathbb{N}\right)$
such that $min\, H_{n}\geq n$ and for each $i\in\left\{ 1,\,2,\,\ldots,\, n\right\} $,
$a+\sum_{t\in H}y_{i,\, t}\in A$. 
\item Given $Y\in\mathcal{Y}$, $J_{Y}=\left\{ p\in e_{S}^{*}\colon\,\forall A\in p,\, A\, is\, a\, J_{Y}-set\, near\, e\right\} $.
\item $J=\cap_{Y\in\mathcal{Y}}J_{Y}$.
\end{enumerate}
\end{defn}
\begin{lem}
\label{lemmaKinJ}Let $\left(T,\,+\right)$ be a commutative semitopological
semigroup and $S$ be a dense subsemigroup of $T$ . Let $\mathcal{B}=\left\{ U_{n}\right\} _{n=1}^{\infty}$
has the \textbf{F }property at the point $e\in E\left(T\right)$.
Let $Y\in\mathcal{Y}$. Then $K\subseteq J_{Y}$. \end{lem}
\begin{proof}
\cite[Lemma 4.2]{akbari2012semigroup}. 
\end{proof}
Due to the above \ref{lemmaKinJ}, the authors, in \cite{akbari2012semigroup}
were able to show that the sets that are central near an idempotent
satisfy a version of the Frustenberg's Central Sets Theorem\textbf{
\cite{furstenberg1981recurrence}} which is as follows.
\begin{thm}
\label{eCntrlThm} Let $\left(T,\,+\right)$ be a semitopological
semigroup and $S$ be a dense subsemigroup of $T$. Let $\mathcal{B}=\left\{ U_{n}\right\} _{n=1}^{\infty}$has
the \textbf{F }property at the point $e\in E\left(T\right)$ and $A$
be a central set near $e$. Suppose that $Y=\left\langle \left\langle y_{i,\, t}\right\rangle _{t=1}^{\infty}\right\rangle _{i=1}^{\infty}\in\mathcal{Y}$
. Then there exist sequences $\left\langle a_{n}\right\rangle _{n=1}^{\infty}$in
$S$ $\left\langle H_{n}\right\rangle _{n=1}^{\infty}$in $\mathcal{P}_{f}\left(\mathbb{N}\right)$
such that 
\begin{enumerate}
\item for each $n\in\mathbb{N}$, $a_{n}\in U_{n}$ and $max\, H_{n}<min\, H_{n+1}$.
\item for each $f\in\Phi$, $\mbox{FS}\left(\left\langle a_{n}+\sum_{t\in H_{n}}y_{f\left(n\right),\, t}\right\rangle _{n=1}^{\infty}\right)\subseteq A$. 
\end{enumerate}
\end{thm}
\begin{proof}
See Theorem 4.11 of \cite{hindman1999semigroup}.
\end{proof}
Akbari Tootkaboni, M and Vahed, T gave an example of one of many possible
combinatorial applications based on \ref{eCntrlThm}.
\begin{cor}
Let $\left(T,\,+\right)$ be a semitopological semigroup and $S$
be a dense subsemigroup of $T$. Let $\mathcal{B}=\left\{ U_{n}\right\} _{n=1}^{\infty}$has
the \textbf{F }property at the point $e\in E\left(T\right)$ and $\left\langle x_{n}\right\rangle _{n=1}^{\infty}$
be a sequence in $S$ such that $\lim_{n\to\infty}x_{n}=e$. Assume
$r\in\mathbb{N}$ and $S=\cup_{i=1}^{r}A_{i}$. Then there is some
$i\in\left\{ 1,\,2,\,\ldots,\, r\right\} $ such that for every $U\in\tau_{e}$
and every $l\in\mathbb{N}$ there is an arithmetic progression $\left\{ a,\, a+d,\,\ldots,\, a+ld\right\} \subseteq A_{i}\cap U$
with increment $d\in FS\left(\left\langle x_{n}\right\rangle _{n=1}^{\infty}\right)$.\end{cor}
\begin{proof}
\cite[Corollary 5.1]{hindman1999semigroup}.
\end{proof}
In our research article, we try to bring out some combinatorial properties
of image partition regularity near an idempotent of a semitopological
semigroup.

\section{Finite Matrices}

First, we shall investigate on some image partition regularity near
$e$ on finite matrices. Then we will move on some infinite ones.
\begin{defn}
Let $S$ be a dense subsemigroup of a semitopological semigroup $\left(T,\,+\right)$.
Let $\mathcal{B}=\left\{ U_{n}\right\} _{n=1}^{\infty}$has the \textbf{F
}property at the point $e\in E\left(T\right)$. Let $M$ be a $u\times v\in\mathbb{N}$
with entries from $\mathbb{Q}$. Then $M$ is image partition regular
over $S$ near $e$ (abbreviated $IPR/S_{e}$) if and only if, whenever
$S=\cup_{i=1}^{r}A_{i}$ for some $r\in\mathbb{N}$, there is some
$i\in\left\{ 1,\,2,\,\ldots,\, r\right\} $ such that for every $U\in\tau_{e}$
there exists $\vec{x}\in S^{v}$ with $\lim_{n\to\infty}x_{n}=e$
for which $M\vec{x}\in A_{i}\cap U$. \end{defn}
\begin{lem}
Let $A$ be a $u\times v$ matrix which is image partition regular
over $\mathbb{N}$. Then it is image partition regular near any $e\in E\left(T\right)$
over any dense subsemigroup $S$ of semitopological semigroup $T$
.\end{lem}
\begin{proof}
Let $r\in\mathbb{N}$ and $S=\bigcup_{j=1}^{r}A_{j}$ . By a standard
compactness arguement pick $n\in\mathbb{N}$ such that whenever $\left\{ 1,\,2,\,3,\,\ldots,\, n\right\} =\bigcup_{j=1}^{r}D_{j}$,
there exist $\overrightarrow{x}\in\left\{ 1,\,2,\,3,\,\ldots,\, n\right\} ^{v}$
and $i\in\left\{ 1,\,2,\,3,\,\ldots,\, r\right\} $ such that $Ax\in\left(D_{i}\right)^{u}$.
Pick 
\[
z\in\left\{ S\bigcap U\colon\, U\in\mathcal{T}_{e}\right\} .
\]
For $i\in\left\{ 1,\,2,\,\ldots,n\right\} $, let 
\[
D_{i}=\left\{ t\in\left\{ 1,\,2,\,3,\,\ldots,\, n\right\} \colon\, z+z+z+z(t-times)\in A_{i}\bigcap U\right\} .
\]
 Thus, picking $i\in\left\{ 1,\,2,\,3,\,\ldots,\, r\right\} $ and
$\overrightarrow{x}\in\left\{ 1,\,2,\,3,\,\ldots,\, n\right\} ^{v}$
such that $A\overrightarrow{x}\in\left(D_{i}\right)^{u}$ and assuming
$\overrightarrow{y}=z\overrightarrow{x}$, we get $A\overrightarrow{y}\in A_{i}\bigcap U$.
\end{proof}

\begin{rem}
\label{Rmk2} This is to note that the converse of the above lemma
also holds good. This can be verified by the Corollary 1.6 of \cite{de2008universally}.
\end{rem}

\begin{defn}
First entries condition. Let $u,\, v\in\mathbb{N}$ and let $A$ be
a $u\times v$ matrix with integer entries. Then 
\begin{enumerate}
\item $A$ satisfies the first entries condition if and only if each row
of $A$ is not $0$ and whenever $i,\, j\in\left\{ 1,\,2,\,\ldots,\, u\right\} $
and 
\[
t=min\left\{ k\in\left\{ 1,\,2,\,\ldots,\, v\right\} \colon\, a_{i,\, k}\neq0\right\} =min\left\{ k\in\left\{ 1,\,2,\,\ldots,\, v\right\} \colon\, a_{j,\, k}\neq0\right\} 
\]
 one has $a_{i,\, t}=a_{j,\, t}>0$. 
\item A number $c$ is a first entry of $A$ if and only if for some $i\in\left\{ 1,\,2,\,\ldots,\, u\right\} $
and some $t\in\left\{ 1,\,2,\,\ldots,\, v\right\} $, $t=min\left\{ k\in\left\{ 1,\,2,\,\ldots,\, v\right\} \colon\, a_{i,\, k}\neq0\right\} $
and $c=a_{i,\, t}$ . 
\end{enumerate}
\end{defn}

\begin{defn}
Let $S$ be a semigroup and let $A\subseteq S$. Then $A$ is an $IP^{*}$set
if and only if for every $IP$ set $B\subseteq S$, $A\cap B\neq\emptyset$.\end{defn}
\begin{lem}
\label{thm:central_img} Let $S$ be a dense subsemigroup of a semitopological
semigroup $\left(T,\,+\right)$. Let $M$ be a $u\times v$ matrix,
$u,\, v\in\mathbb{N}$, with entries from $\omega$ satisfying the
first entries condition. Assume that for every first entry $c$ of
$M$, $cS$ is an $\mbox{IP}^{\star}$-set. Let $C\subseteq S$ be
a $e$-\textup{central} set for any $e\in E\left(T\right)$. Then
there exist sequences $\left\langle x_{1,\, t}\right\rangle _{t=1}^{\infty}$,
$\left\langle x_{2,\, t}\right\rangle _{t=1}^{\infty}$, $\cdots$,
$\left\langle x_{v,\, t}\right\rangle _{t=1}^{\infty}$ in $S$ such
that for each $i\in\left\{ 1,\,2,\,\ldots\,,\, v\right\} $, $lim_{t\to\infty}x_{i,\, t}=e$
and for each $F\in\mathcal{P}_{f}\left(\mathbb{N}\right)$, $M\vec{x}_{F}\in C^{u}$,
where \textup{$\vec{x}_{F}$} is

\[
\begin{array}{ccc}
\vec{x}_{F} & = & \left(\begin{array}{c}
\sum_{F}x_{1,\, n}\\
\sum_{F}x_{2,\, n}\\
\vdots\\
\sum_{F}x_{v,\, n}
\end{array}\right)\end{array}.
\]
\end{lem}
\begin{proof}
Let $v=1$. Then by deleting repeated rows, we get $M=\left(c\right)$
for some $c\in\mathbb{N}$. Given $cS$ being $\mbox{IP}^{\star}$,
$cS\cap C$ is $e$-central set. So, there is some idempotent $p\in e_{S}^{*}$
for which $cS\cap C\in p$. By Theorem 3.2 of \cite{akbari2012semigroup},
there exist some sequence $\left(y_{n}\right)_{n=1}^{\infty}$ in
$S$ such that $\lim_{n\rightarrow\infty}y_{n}=e$, $\sum_{n=1}^{\infty}y_{n}$
converges $e$ and $FS\left(\left(y_{n}\right)_{n=1}^{\infty}\right)\subseteq cS\cap C$.
Since each $y_{n}\in cS\cap C$, we choose $y_{n}=cx_{1,n}$ for each
$n\in\mathbb{N}$ and $x_{1,n}\in S$. Hence, we get our required
$\left(x_{1,n}\right)_{n=1}^{\infty}$ such that $\sum_{n=1}^{\infty}x_{1,\, n}$
converges near $e$ for $v=1$.

Now, we assume that the result holds for $v$ and prove it for $v+1$.
Let $M$ be a $u\times\left(v+1\right)$ matrix satisfying the assumptions
of our result. By adding additional rows, if necessary, we may assume
that there is some $l\in\left\{ 1,\,2,\,\ldots\,,\, u-1\right\} $
and some $c\in\mathbb{N}$ such that for each  $j\in\left\{ 1,\,2,\,\ldots\,,\, u\right\} $,

\[
b_{j,\,1}=\begin{array}{cc}
0 & if\, j\leq l\\
c & if\, j>l
\end{array}
\]

Let $D$ be the $l\times v$ matrix defined by $d_{j,\, i}=b_{j,\, i+1}$
and so $D$ satisfies the first entries condition and all the first
entries of $M$ are the first entries of $D$. Hence, by induction,
choose sequences $\left(x_{1,\, t}\right)_{t=1}^{\infty}$, $\left(x_{2,\, t}\right)_{t=1}^{\infty}$,
$\cdots$, $\left(x_{v,\, t}\right)_{t=1}^{\infty}$ in $S$ such
that $\sum_{t=1}^{\infty}x_{i,\, t}$ converges near $e$ $\forall i\in\left\{ 1,\,2,\,\ldots\,,\, v\right\} $.
For each $j\in\left\{ l+1,\, l+2,\,\ldots\,,\, u\right\} $ and each
$t\in\mathbb{N}$, let $y_{j,\, t}=\sum_{i=2}^{v+1}b_{j,\, i}\cdotp x_{i-1,\, t}$.
Note that since each $\sum_{t=1}^{\infty}x_{i,\, t}$ converges near
$e$ $\forall i\in\left\{ 1,\,2,\,\ldots\,,\, v\right\} $, $\sum_{t=1}^{\infty}y_{j,\, t}$
also converges near $e$. For $j\in\mathbb{N}\setminus\left\{ l+1,\, l+2,\,\ldots\,,\, u\right\} $and
$t\in\mathbb{N}$, let $y_{i,\, t}=y_{u,\, t}$. Let $Y=\left(\left(y_{i,\, t}\right)_{t=1}^{\infty}\right)_{i=1}^{\infty}$
. Note that using the Definition 4.1 of \cite{akbari2012semigroup},
we get $Y\in\mathcal{Y}$. 

Now by applying Theorem 4.3 of \cite{akbari2012semigroup}, choose
sequences $\left(a_{n}\right)_{n=1}^{\infty}$ in $S$ and $\left(H_{n}\right)_{n=1}^{\infty}$
in $\mathcal{P}_{f}\left(\mathbb{N}\right)$ such that 
\begin{enumerate}
\item for each $n\in\mathbb{N}$, $a_{n}\in U_{n}$ and $maxH_{n}<minH_{n+1}$,
and
\item for each $f\in\phi$, $FS\left(\left\langle a_{n}+\sum_{t\in H_{n}}y_{f\left(n\right),\, t}{}_{n=1}^{\infty}\right\rangle \right)\subseteq cS\cap C$. 
\end{enumerate}
By considering the function $f$ to be $f\left(n\right)=n$ for every
$n\in\mathbb{N}$ , we can presume that for each $j\in\left\{ l+1,l+2,\ldots,u\right\} $
and each $F\in\mathcal{P}_{f}\left(\mathbb{N}\right)$ one has $\sum_{n\in F}\left(a_{n}+\sum_{t\in H_{n}}y_{j,\, t}\right)\in cS\cap C$.
For each $n\in\mathbb{N}$, let $z_{1,\, n}$ be the elements satisfying
the equation $a_{n}=z_{1,\, n}+z_{1,\, n}+\cdots z_{1,\, n}$ $\left(c-times\right)$.
For each $n\in\mathbb{N}$ and $i\in\left\{ 2,\,3,\,\ldots\,,\, v+1\right\} ,$
let $z_{i,\, n}=\sum_{t\in H_{n}}x_{i-1,\, t}$. We claim that the
sequences $\left(z_{1,\, t}\right)_{t=1}^{\infty}$, $\left(z_{2,\, t}\right)_{t=1}^{\infty}$,
$\cdots$, $\left(z_{v+1,\, t}\right)_{t=1}^{\infty}$ are as required.

For each $i\in\left\{ 1,\,2,\,\ldots,\, v+1\right\} $, $lim_{t\to\infty}z_{i,\, t}=e$
holds. Let $j\in\left\{ 1,\,2,\,\ldots,\, u\right\} $ and $F\in\mathcal{P}_{f}\left(\mathbb{N}\right)$.
We show that 
\[
\sum_{i=1}^{v+1}b_{j,\, i}\cdotp\sum_{n\in F}z_{i,\, n}\in B
\]

Let $G=\cup_{n\in F}H_{n}$. First, assume $j\in\left\{ 1,2,\ldots,l\right\} .$
Then, 

\[
\sum_{i=1}^{v+1}b_{j,\, i}\cdotp\sum_{n\in F}z_{i,\, n}=\sum_{i=1}^{v+1}b_{j,\, i}\cdotp\sum_{n\in F}\left(\sum_{t\in H_{n}}x_{i-1,\, t}\right)=\sum_{i=1}^{v}d_{j,\, i}\cdotp\sum_{t\in G}x_{i,\, t}\in C
\]
 by induction hypothesis.

Next, let $j\in\left\{ l+1,l+2,\ldots,u\right\} .$ Then

\begin{eqnarray*}
\sum_{i=1}^{v+1}b_{j,\, i}\cdotp\sum_{n\in F}z_{i,\, n} & = & c\cdotp\sum_{n\in F}z_{1,\, n}+\sum_{i=1}^{v+1}b_{j,\, i}\cdotp\sum_{n\in F}\left(\sum_{t\in H_{n}}x_{i-1,\, t}\right)\\
 & = & \sum_{n\in F}\left(a_{n}+\sum_{i=2}^{v+1}b_{j,\, i}\cdotp\sum_{t\in H_{n}}x_{i-1,\, t}\right)
\end{eqnarray*}

and so 
\[
\sum_{i=1}^{v+1}b_{j,\, i}\cdotp\sum_{n\in F}z_{i,\, n}=\sum_{n\in F}\left(a_{n}+\sum_{t\in H_{n}}y_{j,\, t}\right)
\]

\end{proof}

\begin{thm}
\label{IprCntrlEq} Let $M$ be $u\times v$ matrix with entries from
$\omega$. Then the following are equivalent.
\begin{enumerate}
\item $M$ is $IPR/S$.
\item $M$ is $IPR/S_{e}$, $e\in E\left(T\right)$.
\item For every central set $C$ near any $e\in E\left(T\right)$, there
exists $\overrightarrow{x}\in S^{v}$ such that $M\overrightarrow{x}\in C^{u}$.
\end{enumerate}
\end{thm}
\begin{proof}
(b)$\implies$(a) is obvious. (c)$\implies$(b) is obvious too. (b)$\implies$(c)
Given an idempotent $e\in E(T)$, let $C$ be a central set near that
$e$. Let $M$ be $IPR/S_{e}$. By \ref{Rmk2}, we can say $M$ is
$IPR/\mathbb{N}$. Then by a similar result of \textbf{\cite[Theorem 4.1]{hindman2003image}},\textbf{
}there exist $v\times m$ matrix $G$ some $m\in\left\{ 1,\,2,\,\ldots\,,\, u\right\} $
such that $B=MG$ where $B$ is a $u\times m$ first entry matrix.
Then by the above lemma \ref{thm:central_img}, there exist some $\overrightarrow{y}\in S^{m}$
such that $B\vec{y}\in C^{u}$ and hence we get some $\vec{x}\in S^{v}$
such that $M\vec{x}\in C^{u}$.
\end{proof}

\section{Infinite Matrices}

Combinatorial properties of infinite matrices, on the other hand,
is substantially different from that of finite matrices. 
\begin{defn}
Let $S$ be a dense subsemigroup of a semitopological semigroup $\left(T,\,+\right)$.
Let $\mathcal{B}=\left\{ U_{n}\right\} _{n=1}^{\infty}$ has the \textbf{F
}property at the point $e\in E\left(T\right)$. Let $M$ be a $\omega\times\omega$
with entries from $\mathbb{Q}$. Then $M$ is image partition regular
over $S$ near $e$ (abbreviated $IPR/S_{e}$) if and only if, whenever
$S=\bigcup_{i=1}^{r}A_{i}$ for some $r\in\mathbb{N}$, there is some
$i\in\left\{ 1,\,2,\,\ldots,\, r\right\} $ such that for every $U\in\tau_{e}$
there exists $\vec{x}\in S^{v}$ with $\lim_{n\to\infty}x_{n}=e$
for which $M\vec{x}\in A_{i}\cap U$.\end{defn}
\begin{lem}
Let $A$ be finite image partition regular matrix over $\mathbb{N}$
and $B$ be infinite image partition regular matrix near an idempotent
$e$ over any dense subsemigroup $S$ of semitopological semigroup
$T$ . Then 
\[
\left(\begin{array}{cc}
A & O\\
O & B
\end{array}\right)\,
\]
is image partition regular near the idempotent $e$.\end{lem}
\begin{proof}
Let $A$ be a $u\times v$ matrix. Let $r\in\mbox{\ensuremath{\mathbb{N}}}$
and let $\mathbb{N}$ be $r$-colored. By compactness arguement, we
can always assume that whenever the set $\left\{ 1,\,2,\,3,\,\ldots\,,\, n\right\} $
is partitioned into $r$-cells, there exists $\overrightarrow{x}\in\mathbb{N}^{v}$
such that the entries of $A\overrightarrow{x}$ are monochromatic.
Let $S=\bigcup_{j=1}^{r}A_{j}$ be $r$-colored by $\varphi$.

Next, let $S$ be partitioned into $r^{n}$colors via $\psi$, where
$\psi\left(x\right)=\psi\left(y\right)$ if and only if for all $t\in\left\{ 1,\,2,\,3,\,\ldots\,,\, n\right\} $,
$\varphi\left(tx\right)=\varphi\left(ty\right)$. Pick $\overrightarrow{y}\in S^{\omega}$
such that the entries of $B\overrightarrow{y}$ are monochromatic
and are contained in some open subset $U$, $U\in\mathcal{T}_{e}$,
with respect to $\psi$. Pick an entry $a\in\left\{ S\bigcap U\colon\, U\in\mathcal{T}_{e}\right\} $
of $B\overrightarrow{y}$. Define $\gamma\colon\,\left\{ 1,\,2,\,3,\,\ldots,\, n\right\} \to\left\{ 1,\,2,\,3,\,\ldots,\, r\right\} $
by $\gamma\left(i\right)=\varphi\left(ia\right)$. For $i\in\left\{ 1,\,2,\,\ldots,n\right\} $,
let us set $D_{i}=\left\{ t\in\left\{ 1,\,2,\,3,\,\ldots,\, n\right\} \colon\, a+a+\ldots+a(t-\mbox{times})\in A_{i}\bigcap U\right\} $.
Chosing $n\in\mathbb{N}$ large enough so that whenever we partition
the set $\left\{ 1,\,2,\,3,\,\ldots,\, n\right\} $ into $r$ cells
we can pick $i\in\left\{ 1,\,2,\,3,\,\ldots,\, r\right\} $ and $\overrightarrow{x}\in\left\{ 1,\,2,\,3,\,\ldots,\, n\right\} ^{v}$
such that $A\overrightarrow{x}\in\left(D_{i}\right)^{u}$ and assuming
$\overrightarrow{u}=\overrightarrow{x}a$, we get $A\overrightarrow{u}\in A_{i}\bigcap U$.
Choose an entry $i$ of $A\overrightarrow{x}$ and let $j=\gamma\left(i\right)$.

Let $\vec{z}=\left(\begin{array}{c}
a\vec{x}\\
i\vec{y}
\end{array}\right)$. We claim that for any row $\vec{w}$ of $\left(\begin{array}{cc}
A & O\\
O & B
\end{array}\right)$, $\varphi(\vec{w}\cdot\vec{z})=j$. First assume that $\vec{w}$
is a row of $\left(\begin{array}{cc}
M & O\end{array}\right)$, so that $\vec{w}=\left\{ \vec{s}\right\} ^{\frown}\vec{0}$, where
$\vec{s}$ is a row of $M$. Then $\vec{w}\cdot\vec{z}=\vec{s}\cdot(a\vec{x})=a(\vec{s}\cdot\vec{x})$.
Therefore $\varphi(\vec{w}\cdot\vec{z})=\varphi(a(\vec{s}\cdot\vec{x}))=\gamma(\vec{s}\cdot\vec{x})=j$.

Next assume that $\vec{w}$ is a row of $\left(\begin{array}{cc}
O & N\end{array}\right)$, so that $\vec{w}=\vec{0}^{\frown}\vec{s}$ where $\vec{s}$ is a
row of $N$. Then $\vec{w}\cdot\vec{z}=i(\vec{s}\cdot\vec{y})$. Now
$\psi(\vec{s}\cdot\vec{y})=\psi(a)$. So $\varphi(i(\vec{s}\cdot\vec{y}))=\varphi(ia)=\gamma(i)=j$.
\end{proof}

\subsection{Insertion Matrices}

Here, we shall prove a class of infinite matrices, called \emph{Insertion
matrix}, which are image partition regular near an idempotent over
$S$ , where $S$ is a semitopological semigroup. 
\begin{defn}
Let $\gamma,\,\delta\in\omega\cup\left\{ \omega\right\} $ and let
$C$ be a $\gamma\times\delta$ matrix with finitely many nonzero
entries in each row. For each $t<\delta$ , let $B_{t}$ be a finite
matrix of dimension $u_{t}\times v_{t}$ .

Let $R=\left\{ \left(i,\, j\right)\colon\, i<\gamma\,\mbox{and}\, j\in\times_{t<\delta}\left\{ 0,\,1,\,\ldots\,,\, u_{t}-1\right\} \right\} $.
Given $t<\delta$ and $k\in\left\{ 0,\,1,\,\ldots\,,\, u_{t}-1\right\} $,
denote by $\vec{b}_{k}^{\left(t\right)}$ the $k^{th}$ row of $B_{t}$.
Then $D$ is an insertion matrix of $\left\langle B_{t}\right\rangle _{t<\delta}$
into $C$ if and only if the rows of $D$ are all of the form $c_{i,\,0}\cdot\vec{b}_{j\left(0\right)}^{\left(0\right)}\frown c_{i,\,1}\cdot\vec{b}_{j\left(1\right)}^{\left(1\right)}\frown\ldots$\textbf{
}where $\left(i,\, j\right)\in R$.\end{defn}
\begin{example}
If 
\[
C=\left(\begin{array}{cc}
1 & 0\\
2 & 1
\end{array}\right),
\]
 
\[
B_{0}=\left(\begin{array}{cc}
1 & 1\\
5 & 7
\end{array}\right)
\]

and
\[
B_{1}=\left(\begin{array}{cc}
0 & 1\\
3 & 3
\end{array}\right),
\]
 then the insertion matrix is
\[
\left(\begin{array}{cccc}
1 & 1 & 0 & 0\\
1 & 1 & 0 & 0\\
5 & 7 & 0 & 0\\
5 & 7 & 0 & 0\\
2 & 2 & 0 & 1\\
2 & 2 & 3 & 3\\
10 & 14 & 0 & 1\\
10 & 14 & 3 & 3
\end{array}\right)
\]

that is,\textbf{ }
\[
\left(\begin{array}{cccc}
1 & 1 & 0 & 0\\
5 & 7 & 0 & 0\\
2 & 2 & 0 & 1\\
2 & 2 & 3 & 3\\
10 & 14 & 0 & 1\\
10 & 14 & 3 & 3
\end{array}\right)
\]
 is an insertion matrix of $\left\langle B_{t}\right\rangle _{t<2}$
into $C$.
\end{example}
We shall now turn our attention to the Milliken-Taylor matrices with
entries from $\omega$ which are one of the main sources of infinite
image partition regular matrices.
\begin{defn}
Let $\overrightarrow{a}=\left\langle a_{1},\, a_{2},\ldots\,,\, a_{k}\right\rangle $
be a finite sequence in $\mathbb{N}$, and let $\vec{x}=\left\langle x_{n}\right\rangle _{n=1}^{\infty}$
be a sequence in $S$. Then the \emph{Milliken-Taylor system} determined
by $\overrightarrow{a}$ and $\vec{x}$, abbreviated as $\mbox{MT}\left(\overrightarrow{a},\,\vec{x}\right)$
is as follows: 
\[
\left\{ \sum_{t=1}^{k}a_{t}\cdot\sum_{n\in F_{t}}x_{n}:\, F_{t}\in\mathcal{P}_{f}\left(\omega\right),\,\mbox{and\,\ if}\, t<k,\,\mbox{then\,\ max}F_{t}<\mbox{min}F_{t+1}\right\} .
\]
 
\end{defn}
Let $\vec{a}$ has adjacent repeated entries and let $\vec{c}$ is
obtained from $\vec{a}$ by deleting such repetitions, then for any
infinite sequence $\vec{x}$, one has $\mbox{MT}\left(\vec{a},\,\vec{x}\right)\subseteq\mbox{MT}\left(\vec{c},\,\vec{x}\right)$,
so it suffices to consider sequences $\vec{c}$ without adjacent repeated
entries. 
\begin{defn}
Let $\vec{a}$ be a finite or infinite sequence in $\omega$ with
only finitely many nonzero entries. Then $c\left(\vec{a}\right)$
is the sequence obtained from $\vec{a}$ by deleting all zeroes and
then deleting all adjacent repeated entries. The sequence $c\left(\vec{a}\right)$
is the compressed form of $\vec{a}$. If $\vec{a}=c\left(\vec{a}\right)$,
then $\vec{a}$ is a compressed sequence. 
\end{defn}
For example, if $\vec{a}=\left\langle 0,\,1,\,0,\,0,\,1,\,2,\,0,\,2,\,0,\,0,\,\ldots\right\rangle $,
then $c\left(\vec{a}\right)=\left\langle 1,\,2\right\rangle $.
\begin{defn}
Let $\overrightarrow{a}$ be a compressed sequence in $\mathbb{N}.$
A Milliken-Taylor matrix determined by $\overrightarrow{a}$ is $\omega\times\omega$
matrix $A$ such that the rows of $A$ are all possible rows with
finitely many nonzero entries and its compressed form is equal to
$\overrightarrow{a}$.
\end{defn}
Note that the set of entries of $A\vec{x}$ is precisely $\mbox{MT}\left(\vec{a},\,\vec{x}\right)$,
where $A$ is a Milliken-Taylor matrix whose each row have compressed
form $\vec{a}$ and $\vec{x}$ is an infinite sequence in $S$. 
\begin{thm}
Let $(T,+)$ be a commutative semitopological semigroup and $S$ be
a dense subsemigroup of $T$. Let $\mathcal{B}$ have the $\mathbf{F}$
property at the point $e\in E(T)$. Let $\vec{a}=\langle a_{0},a_{1},...,a_{l}\rangle$
be a compressed sequence in $\mathbb{N}$ with $l>0$ , let $C$ be
a Milliken-Taylor matrix for $\vec{a}$ , and for each $t<\omega$
, let $B_{t}$ be a $u_{t}\times v_{t}$ finite matrix with entries
from $\omega$ which is image partition regular matrix over $\mathbb{N}$.
Then any insertion matrix of $\langle B_{t}\rangle_{t<\omega}$ into
$C$ is image partition regular near idempotent ie $IPR/S_{e}$.\end{thm}
\begin{proof}
Pick by Lemma 2.3 of \cite{akbari2012semigroup}some minimal idempotent
$p$ of $e_{S}^{*}$. Let $q=a_{0}\cdot p+a_{1}\cdot p+\cdots+a_{l}\cdot p$.
Then by Lemma 2.5 of \cite{de2008universally}, $q\in e_{S}^{*}$.
Let $\mathcal{G}$ be a finite partition of $S$ and pick $A\in\mathcal{G}$
such that $A\in q$. 

Now $\{x\in S:-x+A\in a_{1}.p+.....+a_{l}.p\}\in a_{0}.p$ so that
$D_{0}=\{x\in S:-a_{0}.x+A\in a_{1}.p+.....+a_{l}.p\}\in p$. Then
$D_{0}^{\star}\in p$ (as used in Theorem 2.6 of \cite{de2008universally}).\textbf{ }

Let $\alpha_{0}=0$ and inductively let $\alpha_{n+1}=\alpha_{n}+v_{n}$.

So pick by theorem \ref{IprCntrlEq} , $x_{0},x_{1},.....,x_{\alpha_{1}-1}\in S$
such that 
\[
B_{0}\left(\begin{array}{c}
x_{0}\\
x_{1}\\
\vdots\\
x_{\alpha_{1}-1}
\end{array}\right)\in\left(D_{0}^{*}\right)^{u_{0}}
\]
Let $H_{0}$ be the set of entries of 
\[
B_{0}\left(\begin{array}{c}
x_{0}\\
x_{1}\\
\vdots\\
x_{\alpha_{1}-1}
\end{array}\right)
\]

Inductively, let $n\in\mathbb{N}$ and assume that we have chosen
$\left\langle x_{t}\right\rangle _{t=0}^{\alpha_{n}-1}$ in $S$ ,
$\left\langle D_{k}\right\rangle _{k=0}^{n-1}$ in $p$ , and $\left\langle H_{k}\right\rangle _{k=0}^{n-1}$
in the set $\mathcal{P}_{f}(\mathbb{N})$ of finite nonempty subsets
of $\mathbb{N}$ such that for $r\in\{0,1,...,n-1\}$,

$\textbf{(I)}$ $H_{r}$ is the set of entries of 
\[
B_{r}\left(\begin{array}{c}
x_{\alpha_{r}}\\
x_{\alpha_{r}+1}\\
\vdots\\
x_{\alpha_{r+1}-1}
\end{array}\right),
\]

$\textbf{(II)}$ if $\emptyset\neq F\subseteq\{0,1,....,r\}$, $k=minF$,
and for each $t\in F$, $y_{t}\in H_{t}$, then $\sum_{t\in F}y_{t}\in D_{k}^{\star}$ 

$\textbf{(III)}$ if $r<n-1$, then $D_{r+1}\subseteq D_{r}$;

$\textbf{(IV)}$ if $m\in\{0,1,....,l-1\}$ and $F_{0},F_{1},....,F_{m}$
are all nonempty subsets of $\{0,1,...,r\}$, for each $i\in\{0,1,...,m-1\}$,
$maxF_{i}<minF_{i+1}$, and for each $t\in\bigcup_{i=0}^{m}F_{i}$,
$y_{t}\in H_{t}$, then $-\sum_{t=0}^{m}a_{i}.\sum_{t\in F_{i}}y_{t}+A\in a_{m+1}.p+a_{m+2}.p+...+a_{l}.p$;

$\textbf{(V)}$ if $r<n-1$ $F_{0},F_{1},....,F_{l-1}$ are nonempty
subsets of the set $\{0,1,...,r\}$, for each $i\in\{0,1,...,m-1\}$,
$maxF_{i}<minF_{i+1}$, and for each $t\in\bigcup_{i=0}^{m}F_{i}$,
$y_{t}\in H_{t}$, then $D_{r+1}\subseteq a_{l}^{-1}(-\sum_{t=0}^{l-1}a_{i}.\sum_{t\in F_{i}}y_{t}+A)$;
and

$\textbf{(VI)}$ if $r<n-1$, $m\in\{0,1,....,l-2\}$, $F_{0},F_{1},....,F_{m}$
are nonempty subsets of $\{0,1,...,r\}$, for each $i\in\{0,1,...,m-1\}$,
$maxF_{i}<minF_{i+1}$, and for each $t\in\bigcup_{i=0}^{m}F_{i}$,
$y_{t}\in H_{t}$, then $D_{r+1}\subseteq\{x\in\mathbb{R}:-a_{m+1}.x+(-\sum_{t=0}^{m}a_{i}.\sum_{t\in F_{i}}y_{t}+A)\in a_{m+2}.p+a_{m+3}.p+...+a_{l}.p\}$.

At $n=1$ , hypotheses $\textbf{(I)}$, $\textbf{(II)}$, and $\textbf{(IV)}$
hold directly while $\textbf{(III)}$, $\textbf{(V)}$, and $\textbf{(VI)}$
are vacuous.

For $m\in\{0,1,...,l-1\}$, let $G_{m}=\{\sum_{t=0}^{m}a_{i}.\sum_{t\in F_{i}}y_{t}$
: $F_{0},F_{1},....,F_{m}$ are nonempty subsets of $\{0,1,...,n-1\}$,
for each $i\in\{0,1,...,m-1\}$, $maxF_{i}<minF_{i+1}$, and for each
$t\in\bigcup_{i=0}^{m}F_{i}$, $y_{t}\in H_{t}\}$.

For $k\in\{0,1,...,n-1\}$, let 

$E_{k}=\{\sum_{t\in F}y_{t}$: if $\phi\neq F\subseteq\{0,1,....,r\}$,
$k=minF$, and for each $t\in F$, $y_{t}\in H_{t}\}$.

Given $b\in E_{k}$, we have that $b\in D_{k}^{\star}$ by hypothesis
$\textbf{(II)}$ and so $-b+D_{k}^{\star}\in p$. If $d\in G_{l-1}$,
then by $\textbf{(IV)}$, $-d+A\in a_{l}.p$ so that $a_{l}^{-1}(-d+A)\in p$.
If $m\in\{0,1,...,l-2\}$ and $d\in G_{m}$, then by $\textbf{(IV)}$,
$-d+A\in a_{m+1}.p+a_{m+2}.p+...+a_{l}.p$ so that

$\{x\in S:-a_{m+1}.x+(-d+A)\in a_{m+2}.p+a_{m+3}.p+...+a_{l}.p\}\in p$.

Thus we have that $D_{n}\in p$, where 

$D_{n}=D_{n-1}\cap\bigcap_{k=0}^{n-1}\bigcap_{b\in E_{k}}(-b+D_{k}^{\star})\cap\bigcap_{d\in G_{l-1}}a_{l}^{-1}(-d+A)$\\
 $\cap\bigcap_{m=0}^{l-2}\bigcap_{d\in G_{m}}\{x\in S:-a_{m+1}.x+(-d+A)\in a_{m+2}.p+a_{m+3}.p+...+a_{l}.p\}$.

(Here, if say $l=1$ or $n<l$ , we are using the convention that
$\bigcap\emptyset=\mathbb{N}$) 

Pick, again by Theorem \ref{IprCntrlEq} , $x_{\alpha_{n}},x_{\alpha_{n}+1},.....,x_{\alpha_{n+1}-1}\in S$
such that 
\[
B_{n}\left(\begin{array}{c}
x_{\alpha_{n}}\\
x_{\alpha_{n}+1}\\
\vdots\\
x_{\alpha_{n+1}-1}
\end{array}\right)\in\left(D_{n}^{*}\right)^{u_{n}}.
\]

Let $H_{n}$ be the set of entries of 
\[
B_{n}\left(\begin{array}{c}
x_{\alpha_{n}}\\
x_{\alpha_{n}+1}\\
\vdots\\
x_{\alpha_{n+1}-1}
\end{array}\right)
\]

Then hypotheses $\textbf{(I)}$, $\textbf{(III)}$, $\textbf{(V)}$,
and $\textbf{(VI)}$ hold directly.

To verify hypothesis $\textbf{(II)}$, let $\emptyset\neq F\subseteq\{0,1,....,r\}$,
let $k=minF$, and for each $t\in F$, let $y_{t}\in H_{t}$. If $n$
does not belong to $F$, then $\sum_{t\in F}y_{t}\in D_{k}^{\star}$
by hypothesis $\textbf{(II)}$ at $n-1$ , so assume that $n\in F$.
If $F=\{n\}$, then we have that $y_{n}\in D_{n}^{\star}$ directly
so assume that $F\neq\{n\}$. Let $b=\sum_{t\in F\setminus\{n\}}y_{t}$.
Then $b\in E_{k}$ and so $y_{n}\in-b+D_{k}^{\star}$ and thus $b+y_{n}\in D_{k}^{\star}$
as required.

To verify hypothesis $\textbf{(IV)}$, let $m\in\{0,1,....,l-1\}$
and $F_{0},F_{1},....,F_{m}$ are nonempty subsets of $\{0,1,...,n\}$,
such that for each $i\in\{0,1,...,m-1\}$, $maxF_{i}<minF_{i+1}$,
and for each $t\in\bigcup_{i=0}^{m}F_{i}$, $y_{t}\in H_{t}$. If
$m=0$ , then $\sum_{t\in F_{0}}y_{t}\in D_{0}^{\star}$ by $\textbf{(II)}$
and $\textbf{(III)}$ so that $-a_{0}.\sum_{t\in F_{0}}y_{t}+A\in a_{1}.p+a_{2}.p+...+a_{l}.p$
as required.

So assume that $m>0$. Let $k=minF_{m}$ and $j=maxF_{m-1}$. Then

$\sum_{t\in F_{m}}y_{t}\in D_{k}^{\star}$ by $\textbf{(II)}$

$\subseteq D_{j+1}$ by $\textbf{(III)}$

$\subseteq\{x\in S:-a_{m}.x+(-\sum_{t=0}^{m-1}a_{i}.\sum_{t\in F_{i}}y_{t}+A)\in a_{m+1}.p+a_{m+2}.p+...+a_{l}.p\}$
by $\textbf{(VI)}$ 

as required.

The induction being complete, we claim that whenever $F_{0},F_{1},....,F_{l}$
are nonempty subsets of $\omega$ such that for each $i\in\{0,1,...,l-1\}$,
$maxF_{i}<minF_{i+1}$, and for each $t\in\bigcup_{i=0}^{m}F_{i}$,
$y_{t}\in H_{t}$, then $-\sum_{i=0}^{l}a_{i}.\sum_{t\in F_{i}}y_{t}\in A$.
To see this , let $k=minF_{l}$ and let $j=maxF_{l-1}$. Then $\sum_{t\in F_{l}}y_{t}\in D_{k}^{\star}\subseteq D_{j+1}\subseteq a_{l}^{-1}(-\sum_{t=0}^{l-1}a_{i}.\sum_{t\in F_{i}}y_{t}+A)$
by hypothesis $\textbf{(V)}$, and so $\sum_{i=0}^{l}a_{i}.\sum_{t\in F_{i}}y_{t}\in A$
as claimed.

Let $Q$ be an insertion matrix of $\langle B_{t}\rangle_{t<\omega}$
into $C$. We claim that all entries of $Q\vec{x}$ are in $A$. To
see this , let $\gamma<\omega$ be given and let $j\in\times_{t<\omega}\{0,1,...,u_{t}-1\}$,
so that

$c_{\gamma,0}.\vec{b_{j(0)}}^{(0)}\frown c_{\gamma,1}.\vec{b}_{j(1)}^{(1)}\frown...$is
a row of $Q$ , say row $\delta$. For each $t\in\{0,1,...,m\}$ ,
let $y_{t}=\sum_{k=0}^{v_{t}-1}b_{j(t),k}^{(t)}.x_{\alpha_{t}+k}$
(so that $y_{t}\in H_{t}$). Therefore we have $\sum_{q=0}^{\infty}q_{\delta,s}.x_{s}=\sum_{t=0}^{m}c_{\gamma,t}.y_{t}$.
Choose nonempty subsets $F_{0},F_{1},....,F_{l}$ of $\{0,1,...,m\}$
such that for each $i\in\{0,1,...,l-1\}$, $maxF_{i}<minF_{i+1}$,
and for each $t\in F_{i}$, $c_{\gamma,t}=a_{i}$. (One can do this
because $C$ is a Milliken -Taylor matrix for $\vec{a}$.) Then $\sum_{t=0}^{m}c_{\gamma,t}.y_{t}=\sum_{i=0}^{l}a_{i}.\sum_{t\in F_{i}}y_{t}\in A$.
\end{proof}

\section{Some infinite Centrally image partition regularity of matrices near
any idempotent }
\begin{defn}
Let $M$ be an $\omega\times\omega$ matrix with entries from $\mathbb{Q}$.
Then $M$ is a segmented image partition regular matrix if and only
if 
\begin{enumerate}
\item no row of $M$ is row is $\vec{0}$; 
\item for each $i\in\omega$, $\{j\in\omega:a_{i,j}\neq0\}$ is finite;
and 
\item there is an increasing sequence $\langle\alpha_{n}\rangle_{n=0}^{\infty}$
in $\omega$ such that $\alpha_{0}=0$ and for each $n\in\omega$,\\
 $\{\langle a_{i,\alpha_{n}},a_{i,\alpha_{n}+1},a_{i,\alpha_{n}+2},\ldots,a_{i,\alpha_{n+1}-1}\rangle:i\in\omega\}\setminus\{\vec{0}\}$\\
 is empty or is the set of rows of a finite image partition regular
matrix.
\end{enumerate}
\end{defn}

If each of these finite image partition regular matrices is a first
entries matrix, then $M$ is a segmented first entries matrix. If
also the first nonzero entry of each $\langle a_{i,\alpha_{n}},a_{i,\alpha_{n}+1},a_{i,\alpha_{n}+2},\ldots,a_{i,\alpha_{n+1}-1}\rangle$,
if any, is 1, then $M$ is a monic segmented first entries matrix.
\begin{defn}
Let $M$ be an $\omega\times\omega$ matrix with entries from $\mathbb{Q}$
and let $S$ be a dense subsemigroup of a semitopological semigroup
$\left(T,\,+\right)$. Given an idempotent $e\in E\left(T\right)$,
we say $M$ is \emph{centrally image partition regular near $e$ }if
and only if whenever $C$ is a central set near $e$ in $S$, there
exists $\vec{x}\in S^{\omega}$ such that $M\vec{x}\in C^{\omega}$.
\end{defn}
Next, we show that segmented image partition regular matrices hold
good for central image partition regularity too.
\begin{thm}
Let $S$ be a dense subsemigroup of $T$ such that for any $c\in\mbox{\ensuremath{\mathbb{N}}},$$cS=\left\{ c-times\, s\colon s\in S\right\} $
is an $IP^{*}$set. Let $M$ be a segmented image partition regular
matrix over $\mathbb{N}$. Then $M$ is \emph{centrally image partition
regular }for any $e-\mbox{central}$ set. \end{thm}
\begin{proof}
Let $\vec{c}_{0},\vec{c}_{1},\vec{c}_{2},\ldots$ denote the columns
of $M$. Let $\langle\alpha_{n}\rangle_{n=0}^{\infty}$ be as in the
definition of a segmented image partition regular matrix. For each
$n\in\omega$, let $M_{n}$ be the matrix whose columns are $\vec{c}_{\alpha_{n}},\vec{c}_{\alpha_{n}+1},\ldots,\vec{c}_{\alpha_{n+1}-1}$.
Then the set of non-zero rows of $M_{n}$ is finite and, if nonempty,
is the set of rows of a finite image partition regular matrix. Let
$B_{n}=(M_{0}$ $M_{1}\ldots M_{n})$.

Lemma 2.3 of \cite{akbari2012semigroup} shows that $e_{S}^{*}$ is
a compact right topological semigroup so that we can choose an minimal
idempotent $p\in e_{S}^{*}$. Let $C\subseteq S$ such that $C\in p$.
Let $C^{*}=\{x\in C:-x+C\in p\}$.Then $C^{*}\in p$. We claim that
$-x+C^{*}\in p$. This can be proved as a consequence of Theorem 4.12
of \cite{hindman1998algebra}. Now the set of non-zero rows of $M_{n}$
is finite and, if nonempty, is the set of rows of a finite image partition
regular matrix over $\mathbb{N}$. As a result of Theorem 15.24 of
\cite{hindman1998algebra}, there exist $m\in\mathbb{N}$ and a $u\times m$
matrix $D$ with entries from $\omega$ which satisfies the first
entries condition such that given any $\overrightarrow{y}\in\mathbb{N}^{m}$
there is some $\overrightarrow{x}\in\mathbb{N}$ with $M_{n}\overrightarrow{x}=D\overrightarrow{y}$.
Then by using Theorem \ref{thm:central_img}, we can choose $\vec{x}^{(0)}\in S^{\alpha_{1}-\alpha_{0}}$
such that, if $\vec{y}=M_{0}\vec{x}^{(0)}$, then $y_{i}\in C^{*}$
for every $i\in\omega$ for which the $i^{th}$ row of $M_{0}$ is
non-zero. We now make the inductive assumption that, for some $m\in\omega$,
we have chosen $\vec{x}^{(0)},\vec{x}^{(1)},\ldots,\vec{x}^{(1)}$
such that $\vec{x}^{(i)}\in S^{\alpha_{i+1}-\alpha_{i}}$ for every
$i\in\{0,1,2,\ldots,m\}$, and, if
\end{proof}
\[
\vec{y}=B_{m}\left(\begin{array}{c}
\vec{x}^{(0)}\\
\vec{x}^{(1)}\\
.\\
.\\
.\\
\vec{x}^{(m)}
\end{array}\right),
\]

then $y_{j}\in C^{*}$ for every $j\in\omega$ for which the $j^{th}$
row of $B_{m}$ is non-zero.

Let $D=\{j\in\omega$ : row $j$ of $B_{m+1}$ is not $\vec{0}\}$
and note that for each $j\in\omega,-y_{j}+C^{*}\in p$. Again by the
previous arguement, we can choose $\vec{x}^{(m+1)}\in S^{\alpha_{m+2}-\alpha_{m+1}}$
such that, if $\vec{z}=M_{m+1}\vec{x}^{(m+1)}$, then $z_{j}\in\bigcap_{t\in D}(-y_{t}+C^{*})$
for every $j\in D$.

Thus we can choose an infinite sequence $\langle\vec{x}^{(i)}\rangle_{i\in\omega}$
such that, for every $i\in\omega$, $\vec{x}^{(i)}\in S^{\alpha_{i+1}-\alpha_{i}}$,
and, if

\[
\vec{y}=B_{i}\left(\begin{array}{c}
\vec{x}^{(0)}\\
\vec{x}^{(1)}\\
.\\
.\\
.\\
\vec{x}^{(i)}
\end{array}\right),
\]

then $y_{j}\in C^{*}$ for every $j\in\omega$ for which the $j^{th}$
row of $B_{i}$ is non-zero.\\

Let 
\[
\vec{x}=\left(\begin{array}{c}
\vec{x}^{(0)}\\
\vec{x}^{(1)}\\
\vec{x}^{(2)}\\
\vdots
\end{array}\right)
\]

and let $\vec{y}=M\vec{x}$. We note that, for every $j\in\omega$,
there exists $m\in\omega$ such that $y_{j}$ is the $j^{th}$ entry
of

\[
B_{i}\left(\begin{array}{c}
\vec{x}^{(0)}\\
\vec{x}^{(1)}\\
.\\
.\\
.\\
\vec{x}^{(i)}
\end{array}\right)
\]

whenever $i>m$. Thus all the entries of $\vec{y}$ are in $C^{*}$.

\bibliographystyle{amsplain}
\bibliography{Paper}

\end{document}